\documentclass[11pt,a4paper,reqno]{amsart} 
\usepackage{latexsym, amsmath, amsfonts, amssymb, amsthm, amscd, epsfig}
\usepackage{url} 
\usepackage[normalem]{ulem} 
\usepackage[utf8]{inputenc} 
\usepackage[colorlinks=false]{hyperref} 
\usepackage{color}

\def \beq {\begin{eqnarray}}
\def \eeq {\end{eqnarray}}
\def \beqn {\begin{eqnarray*}}
\def \eeqn {\end{eqnarray*}}

\newcounter{for}[section]

\newtheorem{itlemma}{Lemma}[section]
\newtheorem{itproposition}[itlemma]{Proposition}
\newtheorem{itfact}[itlemma]{Fact}
\newtheorem{theorem}[itlemma]{Theorem}
\newtheorem{itcorollary}[itlemma]{Corollary}
\newtheorem{itremark}[itlemma]{Remark}
\newtheorem{itremarks}[itlemma]{Remarks}
\newtheorem{itdefinition}[itlemma]{Definition}
\newtheorem{itexample}[itlemma]{Example}

\newenvironment{fact}{\begin{itfact}\rm}{\end{itfact}}
\newenvironment{claim}{\begin{itclaim}\rm}{\end{itclaim}}
\newenvironment{lemma}{\begin{itlemma}}{\end{itlemma}}
\newenvironment{remark}{\begin{itremark}\rm}{\end{itremark}}
\newenvironment{remarks}{\begin{itremarks} \rm}{\end{itremarks}}
\newenvironment{corollary}{\begin{itcorollary}}{\end{itcorollary}}
\newenvironment{proposition}{\begin{itproposition}}{\end{itproposition}}
\newenvironment{definition}{\begin{itdefinition}\rm}{\end{itdefinition}}
\newenvironment{example}{\begin{itexample}\rm}{\end{itexample}}
\newcommand{\be}[1]{\addtocounter{for}{1} \begin{equation}\label{#1}}
\newcommand{\ee}{\end{equation}}
\newcommand{\bl}[1]{\begin{lemma}\label{#1}}
\newcommand{\br}[1]{\begin{remark}\label{#1}}
\newcommand{\brs}[1]{\begin{remarks}\label{#1}}
\newcommand{\bt}[1]{\begin{theorem}\label{#1}}
\newcommand{\bd}[1]{\begin{definition}\label{#1}}
\newcommand{\bp}[1]{\begin{proposition}\label{#1}}
\newcommand{\bfact}[1]{\begin{fact}\label{#1}}
\newcommand{\bc}[1]{\begin{corollary}\label{#1}}
\newcommand{\bex}[1]{\begin{example}\label{#1}}
\newcommand{\ec}{\end{corollary}}
\newcommand{\efact}{\end{fact}}
\newcommand{\eex}{\end{example}}
\newcommand{\el}{\end{lemma}}
\newcommand{\er}{\end{remark}}
\newcommand{\ers}{\end{remarks}}
\newcommand{\et}{\end{theorem}}
\newcommand{\ed}{\end{definition}}
\newcommand{\ep}{\end{proposition}}
\newcommand{\epr}{\end{proof}}
\newcommand{\bpr}{\begin{proof}}
\newcommand{\bcl}[1]{\begin{claim}\label{#1}}
\newcommand{\ecl}{\end{claim}}

\newcommand{\ecs}{\end{corollary}}
\newcommand{\eers}{\end{exercise}}
\newcommand{\eexs}{\end{example}}
\newcommand{\eems}{\end{example}}
\newcommand{\els}{\end{lemma}}
\newcommand{\eles}{\end{lemmaex}}
\newcommand{\ets}{\end{theorem}}
\newcommand{\eds}{\end{definition}}
\newcommand{\eps}{\end{proposition}}

\newcommand{\bi}{\begin{itemize}}
\newcommand{\ei}{\end{itemize}}
\newcommand{\ben}{\begin{enumerate}}
\newcommand{\een}{\end{enumerate}}

\def\vbar{\mathchoice{\vrule height6.3ptdepth-.5ptwidth.8pt\kern-.8pt}
   {\vrule height6.3ptdepth-.5ptwidth.8pt\kern-.8pt}
   {\vrule height4.1ptdepth-.35ptwidth.6pt\kern-.6pt}
   {\vrule height3.1ptdepth-.25ptwidth.5pt\kern-.5pt}}
\def\fudge{\mathchoice{}{}{\mkern.5mu}{\mkern.8mu}}
\def\bbc#1#2{{\rm \mkern#2mu\vbar\mkern-#2mu#1}}
\def\bbb#1{{\rm I\mkern-3.5mu #1}}
\def\bba#1#2{{\rm #1\mkern-#2mu\fudge #1}}
\def\bb#1{{\count4=`#1 \advance\count4by-64 \ifcase\count4\or\bba A{11.5}\or
   \bbb B\or\bbc C{5}\or\bbb D\or\bbb E\or\bbb F \or\bbc G{5}\or\bbb H\or
   \bbb I\or\bbc J{3}\or\bbb K\or\bbb L \or\bbb M\or\bbb N\or\bbc O{5} \or
   \bbb P\or\bbc Q{5}\or\bbb R\or\bbc S{4.2}\or\bba T{10.5}\or\bbc U{5}\or
   \bba V{12}\or\bba W{16.5}\or\bba X{11}\or\bba Y{11.7}\or\bba Z{7.5}\fi}}

\def \R {{\mathbb R}}

\def \N {{\mathbb N}}

\def \ra {\rightarrow }

\def \a{\alpha}
\def \O{\Omega}
\def \s{\sigma}

\def\A{{\mathcal{A}}}
\def\g{\gamma}
\def\d{\delta}
\def\e{\varepsilon}

\def\b{\beta}
\def\L{\Lambda}

\def\1{{\bf 1}}
\def\G{{\Gamma}}

\numberwithin{equation}{section}

\newtheorem{teo}{Theorem}[section]

\newtheorem{de}[teo]{Definition}

\newcommand{\ind}{\mathbf{1}}
\newcommand{\real}{\mathbb{R}}
\newcommand{\integer}{\mathbb{Z}}
\renewcommand{\natural}{\mathbb{N}}

\newcommand{\ie}{\emph{i.e.\ }}

\newcommand{\eg}{\emph{e.g.\ }}

\DeclareMathSymbol{\leqslant}{\mathalpha}{AMSa}{"36} 
\DeclareMathSymbol{\geqslant}{\mathalpha}{AMSa}{"3E} 
\DeclareMathSymbol{\eset}{\mathalpha}{AMSb}{"3F}     
\renewcommand{\leq}{\;\leqslant\;}                   
\renewcommand{\geq}{\;\geqslant\;}                   

\newcommand{\la}{\label} 

\def\1{\ifmmode {1\hskip -3pt \rm{I}} \else {\hbox {$1\hskip -3pt \rm{I}$}}\fi}

\newcommand{\cD}{\ensuremath{\mathcal D}} 
\newcommand{\cE}{\ensuremath{\mathcal E}}

\newcommand{\cL}{\ensuremath{\mathcal L}}

\newcommand{\ent}{{\rm Ent} } 
\newcommand{\var}{{\rm Var} }

\let\a=\alpha \let\b=\beta   \let\d=\delta  \let\e=\varepsilon
 \let\g=\gamma     \let\k=\kappa  
            
  \let\s=\sigma    
  
   \let\G=\Gamma  \let\L=\Lambda 
\let\O=\Omega      

\def\\{\hfill\break}

\def\tthsp{\kern .083333 em}

\def\?{\mskip -10mu}
\def\indbox#1{\hbox to \parindent{\hfil\ #1\hfil} }

\newfam\msafam
\newfam\msbfam
\newfam\eufmfam
\def\hexnumber#1{%
  \ifcase#1 0\or 1\or 2\or 3\or 4\or 5\or 6\or 7\or 8\or
  9\or A\or B\or C\or D\or E\or F\fi}
\font\tenmsa=msam10
\font\sevenmsa=msam7
\font\fivemsa=msam5
\textfont\msafam=\tenmsa
\scriptfont\msafam=\sevenmsa
\scriptscriptfont\msafam=\fivemsa        
\edef\msafamhexnumber{\hexnumber\msafam}%
\mathchardef\restriction"1\msafamhexnumber16
\mathchardef\ssim"0218
\mathchardef\square"0\msafamhexnumber03
\mathchardef\eqd"3\msafamhexnumber2C
\def\QED{\ifhmode\unskip\nobreak\fi\quad
  \ifmmode\square\else$\square$\fi}            
\font\tenmsb=msbm10
\font\sevenmsb=msbm7
\font\fivemsb=msbm5
\textfont\msbfam=\tenmsb
\scriptfont\msbfam=\sevenmsb
\scriptscriptfont\msbfam=\fivemsb
\def\Bbb#1{\fam\msbfam\relax#1}    
\font\teneufm=eufm10
\font\seveneufm=eufm7
\font\fiveeufm=eufm5
\textfont\eufmfam=\teneufm
\scriptfont\eufmfam=\seveneufm
\scriptscriptfont\eufmfam=\fiveeufm

\def\({\left(}
\def\){\right)}
\let\integer=\bZ
\let\real=\bR

\def\ee#1{{\vec {\bf e}_{#1}}}
\let\neper=e
\let\ii=i

\def\ie{\hbox{\it i.e.\ }}

\let\<=\langle
\let\>=\rangle

\outer\def\nproclaim#1 [#2]#3. #4\par{\medbreak \noindent
   \talato(#2){\bf #1 \Thm[#2]#3.\enspace }%
   {\sl #4\par }\ifdim \lastskip <\medskipamount 
   \removelastskip \penalty 55\medskip \fi}
\def\thmm[#1]{#1}
\def\teo[#1]{#1}
\def\sttilde#1{%
\dimen2=\fontdimen5\textfont0
\setbox0=\hbox{$\mathchar"7E$}
\setbox1=\hbox{$\scriptstyle #1$}
\dimen0=\wd0
\dimen1=\wd1
\advance\dimen1 by -\dimen0
\divide\dimen1 by 2
\vbox{\offinterlineskip%
   \moveright\dimen1 \box0 \kern - \dimen2\box1}
}
\begin{document}

\title[Entropy Decay for interacting systems]
{Entropy Decay  for interacting systems via the
  Bochner--Bakry--Emery approach}

\begin{abstract}
 We obtain
 estimates on the exponential rate of decay of the relative entropy
 from equilibrium for Markov processes with a non-local infinitesimal generator. We adapt
 some of the ideas coming from the Bakry-Emery approach to this setting. In particular, we obtain
 volume-independent lower bounds for the Glauber dynamics of interacting point particles and for various classes of hardcore models.
\end{abstract}

\author[P. Dai Pra]{Paolo Dai Pra}
\address{Dipartimento di Matematica Pura e Applicata,
Università di Padova, Via Belzoni 7, 35131 Padova, Italy}
\email{daipra@math.unipd.it}
\author[G. Posta]{Gustavo Posta} \address{Dipartimento di Matematica ``F. Brioschi'', 
Politecnico di Milano, 
P.za Leonardo da Vinci 32, 
20133 Milano, Italy} 
\email{gustavo.posta@polimi.it}


\date{}

\maketitle

\thispagestyle{empty}

\section{Introduction}

The study of contractivity and hypercontractivity of Markov Semigroups has received a tremendous impulse from seminal paper \cite{Ba:Em}, which has introduced the so-called $\G_2$-approach, and has originated a number of developments in different directions (see \eg \cite{GOVW, Le, Li}). In particular, for Brownian diffusions in a convex potential, the $\G_2$-approach provides a short and elegant proof of the fact that lower bounds on the Hessian of the potential translate into lower bounds for both the {\em spectral gap} and the {\em Logarithmic Sobolev constant}. How much these ideas can be adapted to non-local operators, such as generators of discrete Markov Chains, is not yet fully understood. Although $\G_2$-type computations had been performed for specific example (see \eg \cite{Ko:Ku:Oh}), the first attempt to approach systematically this problem appeared in \cite{Bo:Ca:DP:Po}, where lower bounds on the spectral gap of various classes of generators were given. In \cite{Ca:DP:Po} and \cite{Ca:Po} we have addressed the problem of going beyond spectral gap estimates for non-local operators, looking for estimates on the exponential rate of decay of the relative entropy from equilibrium. Note that, in the case of diffusion operators, a strictly positive exponential rate is equivalent to the validity of a Logarithmic Sobolev Inequality. In the non-local case, the exponential entropy decay corresponds to a weaker inequality, to which we will refer to as the {\em Entropy Inequality} (often also called {\em Modified Logarithmic Sobolev Inequality} or $L^1$-{\em Logarithmic Sobolev Inequality}, \cite{wu1}). We have shown in \cite{Ca:DP:Po} that estimates on the best constant in the Entropy Inequality can also be obtained from a $\G_2$-approach; however, when  looking for explicit estimates, we have encountered technical difficulties, that will be illustrated in the next section. More specifically, our results were restricted to particle systems where the only allowed interactions were the exclusion rule (\cite{Ca:DP:Po}) or a {\em zero-range} interaction (\cite{Ca:Po}). 
 
 This paper improves substantially the results mentioned above; we obtain, more specifically, high temperature estimates on the best constant in the Entropy Inequality for Glauber-type dynamics of interacting systems. The main example concerns interacting point particles, where estimates on the spectral gap, as well as constants for other functional inequalities, have been obtained with various techniques \cite{Be:Ca:Ce, Ko:Ku:Oh, Ma:Wa:Wu} . This is however, to our knowledge, the first estimate concerning the Entropy Inequality, that we obtain under the classical {\em Dobrushin Uniqueness Condition}.
 
It should be made clear that the aim of this paper is to extend to non-local operators those implications of the Bakry-Emery's results which are concerned with the rate of convergence to equilibrium of Markov processes. The Bakry-Emery theory has many different, although related, applications, in particular to differential geometry. In this context, extensions to the discrete setting have also been recently considered, see \eg \cite{EM,M}.
 
 The paper is organized as follows. In Section 2 we recall the approach to the spectral gap and entropy decay rate that we have introduced in \cite{Ca:DP:Po}, to which we add the main original ingredient of this paper, consisting in a bivariate real inequality. The rest of the paper is devoted to specific examples.
 
 \section{Generalities}
 
 \subsection{The Entropy Inequality}
 
 We begin by recalling the basic {\em Functional Inequality} we will be concerned with. Consider a time-homogeneous Markov process $(X_t)_{t \geq
  0}$, with values on a measurable space $(S, {\mathcal{S}})$, having
an invariant measure $\pi$. We assume the semigroup $(T_t)_{t \geq 0}$
defined on $L^2(\pi)$ 
by
\[
T_t f(x) := E[f(X_t)|X_0 = x]
\]
is strongly right-continuous, so that the infinitesimal generator
$\cL$ exists, \ie $T_t = e^{t\cL}$. We also assume, for what follows, {\em reversibility} of the process, \ie $\cL$ is self-adjoint in $L^2(\pi)$.
We define the non-negative
quadratic form on $\cD(\cL) \times \cD(\cL)$, called {\em Dirichlet
  form} 
of $\cL$,
\[
\cE(f,g) := - \pi[f \cL g],
\]
where $\cD(\cL)$ is the domain of $\cL$, and we use the notation
$\pi[f]$ for $\int f d\pi$. Given a probability measure $\mu$ on $(S,
{\mathcal{S}})$, we denote by $\mu T_t$ the distribution of $X_t$
assuming $X_0$ is distributed 
according to $\mu$, \ie
\[
\int f d(\mu T_t) := \int (T_t f) d\mu.
\]
An ergodic Markov process, in particular a countable-state, irreducible and recurrent one, has a unique invariant measure $\pi$, and the rate of convergence of $\mu T_t$ to $\pi$ is a major topic of research. Quantitative estimates on this rate of convergence can be obtained by analyzing functional inequalities. To set up the necessary notations, define the {\em relative entropy} $h(\mu|\pi)$ of the probability $\mu$ with respect to $\pi$ by
\[
h(\mu|\pi) := \pi\left[\frac{d\mu}{d\pi} \log \frac{d\mu}{d\pi} \right],
\]
where  $h(\mu|\pi)$ is meant to be infinite whenever $\mu \not\ll \pi$
or $\frac{d\mu}{d\pi} \log \frac{d\mu}{d\pi} \not\in
L^1(\pi)$. Although $h(\cdot \, | \, \cdot)$ is not a metric in the
usual sense, its use as ``pseudo-distance'' is motivated by a number
of relevant properties, the most 
basic ones being:
\[
h(\mu|\pi) = 0 \ \iff \ \mu = \pi
\]
and (see \cite{Di:Sa} equation $(2.8)$)
\be{pinsker}
2\|\mu - \pi\|_{TV}^2 \leq h(\mu|\pi),
\end{equation}
where $\| \cdot \|_{TV}$ denotes the total variation norm. 
For a generic measurable
function $f\geq 0$ it is common to write
\[
\ent_{\pi}(f) := \left\{ \begin{array}{ll} \pi[f \log f] - \pi[f] \log \pi[f] & \mbox{if } f \log f \in L^1(\pi) \\ +\infty & \mbox{otherwise,} \end{array} \right. 
\]
so that $h(\mu|\pi) = \ent_{\pi} \left( \frac{d\mu}{d\pi} \right)$. Ignoring
technical problems concerning the domains of Dirichlet forms, a simple
formal computation shows that
\begin{equation} \label{p1}
\frac{d}{dt} h(\mu T_t | \pi) = - \cE(T_t f, \log T_t f)
\end{equation}
where $f := \frac{d\mu}{d\pi}$. 
Therefore,  assuming that, for each $f \geq 0$, the following {\em Entropy Inequality} ({\bf EI}) holds:
\be{MLSI}
\ent_{\pi}(f) \leq \frac{1}{\a}\, \cE(f, \log  f)
\end{equation}
with $\a >0$ (independent of $f$), then (\ref{p1}) can be closed to
get a differential inequality, 
obtaining
\[
 h(\mu T_t | \pi) \leq e^{- \a t}  h(\mu| \pi).
 \]
In other words, estimates on the best constant $\a$ for which the 
({\bf EI}) holds provide estimates 
for the rate of exponential convergence to equilibrium of the process, 
in the relative entropy sense. It is known  (see \cite{Di:Sa} even though ({\bf EI}) is never explicitly mentioned) that $\a \leq 2\g$, where $\g$ is the {\em spectral gap} for $\cL$:
\be{gap}
\g := \inf \{ \cE(f,f): \ \var_{\pi}(f) := \pi\left[(f-\pi[f])^2 \right] = 1\}.
\end{equation}

\subsection{Convex decay of Entropy}

We now introduce a strengthened version of ({\bf EI}).
Again at a formal level, we compute the second derivative of the entropy along the semigroup:
\begin{equation} \la{2deriv}
\frac{d^{2}}{dt^{2}} \ent_{\pi}(T_t f) = - 
\frac{d}{dt} \cE(T_t f,\log T_t f) = 
\pi\left[ \cL^2 T_t f  \log T_t f\right] + 
\pi \left[ \frac{(\cL T_t f)^2}{T_t f} \right]\,.
\end{equation}
Assume now the inequality
\begin{equation} \label{MLSI'}
\kappa\cE(f, \log f) \leq \pi[\cL^2 f \log f] + \pi\left[ \frac{(\cL f)^2}{f} \right]\,,
\end{equation}
holds for some $\kappa>0$ and every $f>0$.
Then as for the first derivative with ({\bf EI}),  \eqref{2deriv} can be closed to get the differential inequality, 
\begin{equation} \la{2deriv1}
\frac{d}{dt}\cE(T_t f,\log T_t f) \leq -\kappa \cE(T_t f,\log T_t f),
\end{equation}
from which we obtain
\[
\cE(T_t f,\log T_t f) \leq e^{-\k t} \cE(f,\log f).
\]
Rewriting \eqref{2deriv1} as
\[
\frac{d}{dt}\cE(T_t f,\log T_t f)  \leq \kappa \frac{d}{dt} \ent_{\pi}(T_t f)
\]
and integrating from $0$ to $+\infty$ we get
\[
\k \,\ent_{\pi}(f) \leq  \cE(f,\log f)\,.
\]
So \eqref{MLSI'} implies ({\bf EI}) for every $\a \geq \kappa$. This result is well known; however, when one tries to make rigorous the above arguments, some difficulties arise due to the fact that generators are only defined in suitable domains. For this reason we give here the following precise statement: although the assumptions we make are likely to be not optimal, they are sufficient to cover the applications presented in this paper.

\bp{prop:conv-ent}
Assume $\cL$ is self-adjoint in $L^2(\pi)$, and denote by $\cD(\cL)$ its domain of self-adjointness. We write $\cE(f,g) = -\pi[g \cL f]$ whenever $f \in \cD(\cL)$ and $g \in L^2(\pi)$. 
For each $M \in \N$ define
\[
\A_M := \{ f >0, \, f \in \cD(\cL^2), |\log f| \leq M, \cL f \mbox{ is bounded}\}
\]
and assume
$\A_M$ is $L^2(\pi)$-dense in $L^2_M:=\{ f >0, \, f \in L^2, |\log f| \leq M \}$.
Then, setting
\[
 \A:= \bigcup_{M > 0} \A_M,
\]
 the following results hold.
\ben
\item
({\bf EI}) holds for every $f \in \A$ if and only if
\begin{equation} \la{tech1}
\ent_{\pi}(T_t f) \leq e^{-\a t} \ent_{\pi}(f)
\end{equation}
for every $f \geq 0$ measurable, such that $\ent_{\pi}(f) < +\infty$.

\item
\eqref{MLSI'} holds for every $f \in \A$ if and only if
\[
\cE(T_t f,\log T_t f) \leq e^{-\k t} \cE(f,\log f)\,,
\]
for every $f \in \A$.

\item
If (\ref{MLSI'}) holds for some $\k$ and every $f \in \A$, then 
{\rm ({\bf EI})} holds with $\a \geq \k$ and every $f \in \A$.
\een

\ep

%
%
The proof is postponed to the Appendix.  Note that (\ref{MLSI'}) gives estimates on the second derivative of the entropy along the flow of the semigroup $T_t$. In particular, being $ \cE(f,\log f)\geq0$, it implies time convexity of the entropy. 
There are cases (see \cite{Ca:DP:Po} Section~4.2) where ({\bf EI}) holds but the entropy is non convex in time.
Therefore, (\ref{MLSI'}) is {\em strictly} stronger that ({\bf EI}).

\br{rem:spec}
By a similar proof one shows that the spectral gap $\g$ is the best constant in the inequality
\begin{equation} \label{gap-conv}
k\cE(f,f) \leq \pi\left[(\cL f)^2 \right],
\end{equation}
that is {\em equivalent} to the Poincaré inequality
\[
k \var_{\pi}(f) \leq \cE(f,f),
\]
whose best constant is, by definition, the spectral gap of $\cL$. Inequality \eqref{gap-conv} is related to the convex decay of the variance along the flow of the semigroup. Unlike the entropy, the variance decay is always convex in time.
\er

%

\subsection{A class of non-local dynamics}

Suppose the probability space $(S, \mathcal{S},\pi)$ is given,
together with a set $G$ of measurable functions $\g:S \ra S$, that we
call {\em moves}. We also assume $G$ is provided with a measurable
structure, \ie a $\s$-algebra $\mathcal{G}$ of subsets of $G$. In this
paper we deal with Markov generators that, can be written in the form
\begin{equation} \label{gen}
\cL f(\eta) = \int_G \nabla_{\g}f(\eta) c(\eta,d\g),
\end{equation}
where
\bi
\item
the {\em discrete gradient} $\nabla_{\g}$ is defined by
\[
\nabla_{\g}f(\eta) := f(\g(\eta)) - f(\eta);
\]
\item
for $\eta \in S$, $c(\eta,d\g)$ is a positive, finite measure on
$(G,\mathcal{G})$, such that for each $A \in \mathcal{G}$ the map
$\eta \mapsto c(\eta,A)$ is measurable, and  $\pi[c^2(\eta,G)] < +\infty$.

\ei
It should be stressed that not necessarily an expression as in \eqref{gen} defines a Markov generator. We assume this is the case. We make the following additional assumption on the generator $\cL$.

\vspace{0.5cm}
\noindent
{\bf \em (Rev)} There is a measurable involution
\begin{eqnarray*}
G & \ra & G \\ \g & \mapsto & \g^{-1}
\end{eqnarray*}
such that the equality $\g^{-1}(\g(\eta)) = \eta$ holds $c(\eta,d\g) \pi(d\eta)$-almost everywhere. Moreover, for every $\Psi: S \times G \ra \R$ measurable and bounded, 
\be{db}
\int \Psi(\eta,\g) c(\eta,d\g) \pi(d\eta) = \int \Psi(\g(\eta),\g^{-1}) c(\eta,d\g) \pi(d\eta).
\end{equation}
Note that, since $\pi[c^2(\eta,G)] < +\infty$ and letting $\cD_0$ be the set of bounded, measurable functions from $S$ to $\R$, we have that $\cL f \in L^2(\pi)$ for $f \in \cD_0$. Moreover, by {\bf \em (Rev)}, $\cL$ is symmetric on $\cD_0$ and, for $f,g \in \cD_0$, 
\be{dirich}
\cE(f,g) = \cE(g,f) = \frac{1}{2}\int_G\pi\left[c(\cdot,d\g) \nabla_{\g} f \nabla_{\g} g  \right].
\end{equation}
In particular \eqref{dirich} implies that $-\cL$ is a positive operator so, by considering its Friedrichs extension, $\cL$ can be extended to a domain of self-adjointness $\cD(\cL) \supseteq \cD_0$. It also follows that if $f \in \A$, where $\A$ has been defined in Proposition \ref{prop:conv-ent}, then $\log f \in \cD(\cL)$, and
\[
\pi\left[ \cL^2 T_t f  \log T_t f\right] = \pi\left[ \cL T_t f  \cL \log T_t f\right] = \int\pi\left[ c(\cdot,d\g) c(\cdot,d\d)\nabla_{\g} f \nabla_{\d}\log  f  \right]. 
\]
\bd{admissible}
A finite measure $R$ on $S \times G \times G$ is said {\em admissible} if the following conditions hold.
\bi
\item[i)]
$R$ is supported on the set $\{(\eta,\g,\d): \g(\d(\eta)) = \d(\g(\eta))\}$.
\item[ii)]
The maps $(\eta,\g,\d) \mapsto (\eta,\d,\g)$ and $(\eta,\g,\d) \mapsto (\g(\eta), \g^{-1}, \d)$ are $R$-preserving.
\ei
Similarly, we say that a nonnegative measurable function $r:S \times G \times G \ra [0,+\infty)$ is {\em admissible} if the measure $R(d\eta,d\g,d\d) := c(\eta,d\g)c(\eta,d\d)r(\eta,\g,\d) \pi(d\eta)$ is admissible.
\ed
By {\bf \em (Rev)}, it is easy to check that a function $r \in L^1( c(\eta,d\g)c(\eta,d\d)\pi(d\eta))$ is admissible if the following conditions hold:
\bi
\item[a)]
$r$ is supported on the set $\{(\eta,\g,\d): \g(\d(\eta)) = \d(\g(\eta))\}$, up to sets of zero measure for $ c(\eta,d\g)c(\eta,d\d)\pi(d\eta)$;
\item[b)]
the following equality holds $c(\eta,d\g)c(\eta,d\d)\pi(d\eta)$-almost everywhere:
\[
r(\eta,\g,\d) = r(\eta,\d,\g).
\]
\item[c)]
the equality (between measures on $G$)
\be{admissibility}
c(\eta,d\d) r(\eta,\g,\d) = c(\g(\eta), d\d) r(\g(\eta), \g^{-1},\d)
\end{equation}
holds $c(\eta,d\g) \pi(d\eta)$-almost everywhere.
\ei

Admissible measures guarantee the following Bochner-type identities. A proof of these identities is in \cite{Ca:DP:Po}; we include it here for completeness.
\bp{prop:Bochner}
The following identities hold for every bounded measurable functions $f,g:S \ra \R$:
\be{boch1}
\int \nabla_{\g}f(\eta) \nabla_{\d} g(\eta) R(d\eta,d\g,d\d) = \frac{1}{4} \int \nabla_{\g} \nabla_{\d} f(\eta) \nabla_{\g} \nabla_{\d} g(\eta) R(d\eta,d\g,d\d) ,
\end{equation}
\begin{multline} \label{boch2}
\int \frac{\nabla_{\g}f(\eta) \nabla_{\d} f(\eta)}{f(\eta)} R(d\eta,d\g,d\d)  \\ =  \frac{1}{4} \int \left[\nabla_{\g}\left(\frac{\nabla_{\d}f(\eta)}{f(\d(\eta))}\right)\nabla_{\g} \nabla_{\d} f(\eta) - \nabla_{\g}\left(\frac{(\nabla_{\d}f(\eta))^2}{f(\eta)f(\d(\eta))}\right) \nabla_{\g}f(\eta)\right] R(d\eta,d\g,d\d) .
\end{multline}
\ep
\bpr
We begin by proving \eqref{boch1}. By i) of Definition \ref{admissible},  
\[
\nabla_{\g}\nabla_{\d} f(\eta ) \nabla_{\g}\nabla_{\d}g(\eta ) = \nabla_{\g}\nabla_{\d} f(\eta ) \nabla_{\d}\nabla_{\g}g(\eta )
\]
 $R$-almost everywhere. Thus, $R$-almost everywhere,
\begin{align} 
& \nabla_{\g}\nabla_{\d} f(\eta ) \nabla_{\d}\nabla_{\g}g(\eta ) \nonumber \\
&\quad= \nabla_{\g}f(\d(\eta ))\nabla_{\d}g(\g(\eta )) - \nabla_{\g} f(\d(\eta )) \nabla_{\d}g(\eta )  - \nabla_{\g} f(\eta ) \nabla_{\d}g(\g(\eta )) + \nabla_{\g} f(\eta ) \nabla_{\d}g(\eta )\,. \label{2der}
 \end{align}
 We show that the $R$-integral of each summand of \eqref{2der} equals $\int \nabla_{\g}f(\eta) \nabla_{\d} g(\eta) R(d\eta,d\g,d\d)$, from which \eqref{boch1} follows. For the fourth summand there is nothing to prove. In the steps that follow we use admissibility of $R$, in particular first  ii), then i), then ii) and i) again of Definition \ref{admissible}, and
 the simple identity $\nabla_{\g}f(\eta) = - \nabla_{\g^{-1}}f(\g(\eta))$:
\begin{align} 
 \int \nabla_{\g}f(\eta) \nabla_{\d} g(\eta) R(d\eta,d\g,d\d) & =  \int \nabla_{\g^{-1}}f(\g(\eta)) \nabla_{\d} g(\g(\eta)) R(d\eta,d\g,d\d) \nonumber \\ &  = - \int  \nabla_{\g} f(\eta ) \nabla_{\d}g(\g(\eta ))R(d\eta,d\g,d\d) \label{3rd}
  \\ &  = - \int  \nabla_{\d} f(\eta ) \nabla_{\g}g(\d(\eta ))R(d\eta,d\g,d\d) \nonumber \\ 
 & = - \int  \nabla_{\d} f(\g(\eta) ) \nabla_{\g^{-1}}g(\d(\g(\eta) ))R(d\eta,d\g,d\d) \nonumber \\ 
 & = \int  \nabla_{\d} f(\g(\eta) ) \nabla_{\g} g(\d(\eta)) R(d\eta,d\g,d\d) \nonumber \\
 & =  \int \nabla_{\g}f(\d(\eta ))\nabla_{\d}g(\g(\eta )) R(d\eta,d\g,d\d). \label{1st}
 \end{align}
 Note that \eqref{3rd} takes care of the third (and by symmetry the second) summand, while \eqref{1st} takes care of the first summand. This completes the proof of \eqref{boch1}.
 
 We now prove \eqref{boch2}. By admissibility of $R$ (used twice), 
 \begin{align*}
 \int \frac{\nabla_{\g}f(\eta) \nabla_{\d} f(\eta)}{f(\eta)} R(d\eta,d\g,d\d)   & = - \int  \frac{\nabla_{\g}f(\d^{-1} \d(\eta)) \nabla_{\d^{-1}} f(\d(\eta))}{f(\d^{-1} \d(\eta))} R(d\eta,d\g,d\d)  
 \\ & = - \int \frac{\nabla_{\g}f(\d(\eta)) \nabla_{\d} f(\eta)}{f(\d(\eta))} R(d\eta,d\g,d\d) \\ 
 & =  \int \frac{\nabla_{\g}f(\d(\eta)) \nabla_{\d} f(\g(\eta))}{f(\g \d(\eta))} R(d\eta,d\g,d\d) .
 \end{align*}
 Thus
 \begin{multline*}
  \int \frac{\nabla_{\g}f(\eta) \nabla_{\d} f(\eta)}{f(\eta)} R(d\eta,d\g,d\d) \\ = \frac{1}{4} \left[ \int \frac{\nabla_{\g}f(\eta) \nabla_{\d} f(\eta)}{f(\eta)} R(d\eta,d\g,d\d) - \int \frac{\nabla_{\g}f(\d(\eta)) \nabla_{\d} f(\eta)}{f(\d(\eta))} R(d\eta,d\g,d\d) \right. \\ \left. +  \int \frac{\nabla_{\g}f(\d(\eta)) \nabla_{\d} f(\g(\eta))}{f(\g \d(\eta))} R(d\eta,d\g,d\d) - \int \frac{\nabla_{\g}f(\d(\eta)) \nabla_{\d} f(\eta)}{f(\d(\eta))} R(d\eta,d\g,d\d) \right]
  \end{multline*}
  that, by a simple calculation, is shown to equal the right hand side of \eqref{boch2}.

 \epr
 
 The use of admissible measures in establishing convex entropy decay is illustrated in what follows. Consider the inequality \eqref{MLSI'}; the two sides if the inequality, for generators of the form \eqref{gen} take the form
 \begin{align}
 \cE(f, \log f) & = \frac{1}{2}\pi\left[ \int c(\eta,d\g) \nabla_{\g} f(\eta) \nabla_{\g}\log f(\eta)  \right]  \label{lhs} \\
  \pi[\cL f \cL\log f] + \pi\left[ \frac{(\cL f)^2}{f} \right] & = \int\pi\left[ c(\cdot,d\g) c(\cdot,d\d) \left( \nabla_{\g} f \nabla_{\d}\log  f + \frac{ \nabla_{\g} f \nabla_{\d} f}{f} \right) \right]. \label{rhs}
  \end{align}
Admissible measures allow to modify the term \eqref{rhs}, the purpose being to make it comparable with \eqref{lhs}.
\bp{gamma}
Let $R$ be an admissible measure. Then for every $f>0$ measurable with $\log f$ bounded,
\be{gamma1}
\int R(d\eta,d\g,d\d)  \left( \nabla_{\g} f(\eta) \nabla_{\d}\log  f(\eta) + \frac{ \nabla_{\g} f(\eta) \nabla_{\d} f(\eta)}{f(\eta)} \right) \geq 0.
\end{equation}
Therefore, letting $\Gamma(d\eta,d\g,d\d) := \pi(d\eta) c(\eta,d\g) c(\eta,d\d) - R(d\eta,d\g,d\d)$, we have
\begin{multline} \label {gamma2}
\pi\left[ \int c(\eta,d\g) c(\eta,d\d) \left( \nabla_{\g} f(\eta) \nabla_{\d}\log  f(\eta) + \frac{ \nabla_{\g} f(\eta) \nabla_{\d} f(\eta)}{f(\eta)} \right) \right] \\ \geq \int \G(d\eta,d\g,d\d)  \left( \nabla_{\g} f(\eta) \nabla_{\d}\log  f(\eta) + \frac{ \nabla_{\g} f(\eta) \nabla_{\d} f(\eta)}{f(\eta)} \right).
\end{multline}
\ep
\bpr
Using \eqref{boch1} with $g = \log f$, we get
\[
\int R(d\eta,d\g,d\d)  \nabla_{\g} f(\eta) \nabla_{\d}\log  f(\eta) = \frac{1}{4} \int R(d\eta,d\g,d\d) \nabla_{\g} \nabla_{\d} f(\eta)  \nabla_{\g} \nabla_{\d}\log f(\eta) 
\]
Thus, using also \eqref{boch2},
\begin{multline}
\int R(d\eta,d\g,d\d)  \left( \nabla_{\g} f(\eta) \nabla_{\d}\log  f(\eta) + \frac{ \nabla_{\g} f(\eta) \nabla_{\d} f(\eta)}{f(\eta)}\right)  = \\ \frac{1}{4}  \int R(d\eta,d\g,d\d) \left[  \nabla_{\g} \nabla_{\d} f(\eta)  \nabla_{\g} \nabla_{\d}\log f(\eta) + \nabla_{\g}\left(\frac{\nabla_{\d}f(\eta)}{f(\d(\eta))}\right)\nabla_{\g} \nabla_{\d} f(\eta) \right. \\ \left.  - \nabla_{\g}\left(\frac{(\nabla_{\d}f(\eta))^2}{f(\eta)f(\d(\eta))}\right) \nabla_{\g}f(\eta)\right] \label{gamma2-1}
\end{multline}
The fact that \eqref{gamma2-1} is nonnegative, follows from the nonnegativity of 
\[
\nabla_{\g} \nabla_{\d} f(\eta)  \nabla_{\g} \nabla_{\d}\log f(\eta) + \nabla_{\g}\left(\frac{\nabla_{\d}f(\eta)}{f(\d(\eta))}\right)\nabla_{\g} \nabla_{\d} f(\eta)   - \nabla_{\g}\left(\frac{(\nabla_{\d}f(\eta))^2}{f(\eta)f(\d(\eta))}\right) \nabla_{\g}f(\eta)
\]
for {\em every} $\eta,\g,\d$. Indeed, setting $a:=f(\eta ), b:=f(\d(\eta )), c := f(\g(\eta )), d := f(\d \g(\eta ))$, one checks that this last expression equals the sum of the following $4$ expressions
\[
\begin{array}{c}
d\log d - d \log(bc/a) + (bc/a) - d \\ c\log c - c \log(da/b) + (da/b) -c \\ b\log b - b \log(da/c) + (da/c) - b \\ a \log a - a \log(bc/d) + (bc/d) -a
\end{array}
\]
which are all non-negative, since $\a \log \a - \a \log \b + \b - \a
\geq 0$ for every $\a,\b >0$. The proof is therefore completed.

\epr

By \eqref{lhs}, \eqref{rhs} and Proposition \ref{gamma}, convex decay of entropy, \ie inequality \eqref{MLSI'} follows by showing
\be{final}
\int \Gamma(d\eta,d\g,d\d)  \left( \nabla_{\g} f(\eta) \nabla_{\d}\log  f(\eta) + \frac{ \nabla_{\g} f(\eta) \nabla_{\d} f(\eta)}{f(\eta)} \right) \geq  \frac{\kappa}{2}  \int\pi\left[c(\cdot,d\g) \nabla_{\g} f \nabla_{\g}\log f  \right]
\end{equation}
for every $f>0$ measurable, with $\log f$ bounded.
To illustrate the treatment of \eqref{final}, we consider the corresponding inequality for the spectral gap studied in \cite{Bo:Ca:DP:Po}:
\be{eq:gap}
\int \Gamma(d\eta,d\g,d\d) \nabla_{\g} f(\eta) \nabla_{\d}  f(\eta) \geq \frac{k}{2} \int\pi\left[ c(\cdot,d\g)\left( \nabla_{\g} f\right)^2 \right].
\end{equation}
The strategy to obtain \eqref{eq:gap} can be described in two steps.
\bi
\item[i)]
Determine an admissible function $r(\eta,\g,\d)$ and a ``nearly diagonal'' $D \subseteq G \times G$ such that
\begin{multline} \label{eq:gap1}
\int_D \Gamma(d\eta,d\g,d\d) \nabla_{\g} f(\eta) \nabla_{\d}  f(\eta) =
\int_D\pi\left[  c(\cdot,d\g)c(\cdot,d\d) [1-r(\cdot,\g,\d)] \nabla_{\g} f \nabla_{\d}  f \right] \\
\geq u  \int\pi\left[ c(\cdot,d\g)\left( \nabla_{\g} f\right)^2 \right]
\end{multline}
for some $u>0$.
\item[ii)]
The remaining integral on $D^c$ is estimated from below using the inequality $2ab \geq -a^2 - b^2$ which, by symmetry, yields
\begin{multline} \label{eq:gap2}
\pi\left[ \int_{D^c} c(\eta,d\g)c(\eta,d\d) [1-r(\eta,\g,\d)] \nabla_{\g} f(\eta) \nabla_{\d}  f(\eta) \right]  \geq\\  - \int_{D^c}\pi\left[ c(\cdot,d\g)c(\cdot,d\d) |1-r(\cdot,\g,\d)| \left(\nabla_{\g} f \right)^2 \right] \geq - h  \int \pi\left[ c(\cdot,d\g)\left( \nabla_{\g} f\right)^2 \right],
\end{multline}
where
\be{eq:gap3}
h := \sup_{\eta,\g} \int_{\{\d: (\g,\d) \in D^c\}} c(\eta,d\d) |1-r(\eta,\g,\d)|.
\end{equation}
\ei
If $h<u$, we thus obtain \eqref{eq:gap} with $k:= 2(u-h)$. 

The feasibility of steps i) and ii) above depends on a suitable choice of an admissible function $r$. We do not have a general procedure to determine it. It turns out, for example, that equation \eqref{admissibility} does not uniquely (up to constant factors) determine $r$. Condition \eqref{admissibility} is, for instance, satisfied by
\be{canonical}
r(\eta,\g,\d) := \frac{1}{2} \left[\frac{c(\g(\eta),d\d)}{c(\eta,d\d)} +1\right],
\end{equation}
which is well defined whenever the Radon-Nikodym derivative $\frac{c(\g(\eta),d\d)}{c(\eta,d\d)}$ exists. Not necessarily, however, \eqref{canonical} defines an admissible function, in particular it is not necessarily supported on the set $\{(\eta,\g,\d): \g(\d(\eta)) = \d(\g(\eta))\}$. The admissible functions in the examples in \cite{Bo:Ca:DP:Po}, \cite{Ca:DP:Po} as well as those in this paper, are all obtained by suitable modifications of \eqref{canonical}.

The main purpose of this paper is to extend the procedure above to inequality \eqref{final}. The main difficulty consists in the comparison of the ``off diagonal terms''
\[
\int_{D^c}\pi\left[  c(\cdot,d\g)c(\cdot,d\d) [1-r(\cdot,\g,\d)] \left( \nabla_{\g} f \nabla_{\d}\log  f + \frac{ \nabla_{\g} f \nabla_{\d} f}{f} \right) \right]
\]
with corresponding diagonal terms (\ie $\d = \g$). The simple inequality $2ab \geq -a^2 - b^2$ is the replaced by the following inequality.
\bl{lemma:key}
The following inequality holds for every $a,b>0$:
\begin{multline} \label{key}
(a-1) \log b + (b-1) \log a + 2(a-1)(b-1) \\ \geq 
- \left[ 
(a-1) \log a + (
b-1) \log b  + 
\frac{(a-1)^2}{a} + \frac{(b-1)^2}{b}
\right].
\end{multline}
\el
\bpr
Inequality \eqref{key} can be rewritten as
 \begin{equation}
  (a+b-2) \log(ab) + 2ab - (a+b) - 2 + \frac{a+b}{ab} \geq 0 . \label{key3}
 \end{equation}
 Letting $z := a+b$, $w=ab$, we are left to show that for $z,w>0$
 \begin{equation}\label{key4}
 (z-2) \log w +2 w - z - 2 + \frac{z}{w} \geq 0.
 \end{equation}
 
 \noindent
 {\em Case $z \geq 2$}.
 Using the inequality $\log(1+x) \leq x$ for every $x>-1$,
 \[
 (z-2) \log w = -(z-2) \log\left(1+ \frac{1-w}{w} \right) \geq -(z-2)\frac{1-w}{w}.
 \]
 Thus
\[
 (z-2) \log w + 2 w - z - 2 + \frac{z}{w} \geq  2\left(w + \frac{1}{w} - 2 \right) \geq 0.
 \]

  \noindent
 {\em Case $z < 2$}.  Using again $\log(1+x) \leq x$ for every $x>-1$,
 \[
 (z-2) \log w = (z-2) \log[1+(w-1)] \geq (z-2)(w-1),
 \]
 so
 \[
 (z-2) \log w + 2 w - z - 2 + \frac{z}{w} \geq  z\left(w + \frac{1}{w} - 2 \right) \geq 0.
 \]
\epr
Letting 
\[
a:= \frac{f(\g(\eta))}{f(\eta)} \ \ b:=  \frac{f(\d(\eta))}{f(\eta)},
\]
\eqref{key} becomes
\begin{multline} \label{ineqkey}
\nabla_{\g} f(\eta) \nabla_{\d} \log f(\eta) + \nabla_{\d} f(\eta) \nabla_{\g} \log f(\eta) + 2 \frac{\nabla_{\g} f(\eta) \nabla_{\d} f(\eta)}{f(\eta)} \\ \geq - \nabla_{\g} f(\eta) \nabla_{\g} \log f(\eta) - \nabla_{\d} f(\eta) \nabla_{\d} \log f(\eta) - \frac{\left(\nabla_{\g} f(\eta)\right)^2}{f(\g(\eta))} - \frac{\left(\nabla_{\d} f(\eta)\right)^2}{f(\d(\eta))}.
\end{multline}

\section{Examples} \label{sec:examples}

\subsection{Glauber dynamics of particles in the continuum} \label{sec:gas}
\label{sec:cont}
Let $\O$ be the set of locally finite subsets of $\R^d$. We provide $\O$ with the weakest topology that, for every continuous $f :\R^d \ra \R$ with compact support, makes the maps $\eta \mapsto \sum_{x \in \eta} f(x)$ continuous. Measurability on $\O$ is provided by the corresponding Borel $\s$-field.

Now let $\L$ be a bounded Borel subset of $\R^d$ of nonzero Lebesgue measure, and set
\[
S := \O_{\L} := \{\eta \in \O: \eta \subseteq \L\}.
\]
Consider a nonnegative measurable and even function $\varphi:\R^d \ra [0,+\infty)$ (everything works with minor modifications for $\varphi:\R^d \ra [0,+\infty]$ allowing ``hardcore repulsion''). We fix a {\em boundary condition} $\tau \in \O_{\L^c} := \{ \eta \in \O: \eta \subseteq \L^c\}$, and define the Hamiltonian $H^{\tau}_{\L}:S \ra [0,+\infty]$
\begin{equation}
H^{\tau}_{\L}(\eta) = \sum_{\stackrel{\{x,y\} \subseteq \eta \cup \tau}{\scriptscriptstyle \{x,y\} \cap \L \neq \emptyset}} \varphi(x-y).
\label{hamcont}
\end{equation}
The dependence of $H^{\tau}_{\L}$ on $\L$ and $\tau$ is omitted in the sequel.
We assume the nonnegative pair potential $\varphi$ and the inverse temperature $\b$ to satisfy 
 the condition
\be{integrability}
\e(\b):= \int_{\R^d} \left(1-e^{-\b \varphi(x)}\right)dx < +\infty.
\end{equation}
For $N\in\N$ we let $S_N=\{\eta \in S : |\eta| = N\}$
denote the subset of $S$ consisting of all possible configurations
of $N$ particles in $\Lambda$. Note that a measurable function $f:S_N \ra \R$ may be identified with a symmetric function from $\L^N \ra \R$. With this identification, we assume, for every  $N\in\N$, that the boundary condition $\tau$ is such that $H(\eta) < +\infty$ in a subset of $\L^N$ having positive Lebesgue measure. Functions from $S$ to $\R$ may be identified with symmetric functions from $\bigcup_n \L^n$ to $\R$. With this identification, we define the finite volume {\em grand canonical} Gibbs measure $\pi$ with inverse temperature $\b>0$ and activity $z>0$ by
\begin{equation}
\pi[f] := \frac{1}{Z} \sum_{n=0}^{+\infty} \frac{z^n}{n!} \int_{\L^n}
e^{-\b H(x)}f(x)\,dx,
\label{nuf}
\end{equation}
where $Z$ is the normalization. We define the creation and annihilation maps on $S$: for $x \in \L$
\[
\g^+_x(\eta) = \eta \cup \{x\},
\qquad\qquad
\g^-_x(\eta) = \eta \setminus \{x\}.
\]

We let $G:= \{\g^+_x,\g^-_x: x \in \L\}$. In the sequel we write $\nabla^+_x$ and $\nabla^-_x$ rather than $\nabla_{\g^+_x}$ and $\nabla_{\g_x^-}$. Note that $\nabla^-_x f(\eta) = 0$ unless $x \in \eta$. We consider the following Markov generator
\be{genglac}
\cL f(\eta) := \sum_{x \in \eta} \nabla^-_x f(\eta) + z \int_{\L} e^{-\b\nabla^+_x H(\eta)} \nabla^+_x f(\eta)\,dx.
\end{equation}
It is shown in \cite{Be:Ca:Ce}, Proposition 2.1, that $\cL$ generates a Markov semigroup. 
This generator is 
 of the form (\ref{gen}) if we define $c(\eta,d\g)$ by
\[
\int F(\g)c(\eta,d\g) := \sum_{x \in \eta} F(\g^-_x) + z \int_{\L} e^{-\b\nabla^+_x H(\eta)} F(\g^+_x)\,dx.
\]
In particular, it is easy to show that the reversibility condition (\ref{db}) holds, after having observed that
\[
\left(\g^+_x\right)^{-1} = \g^-_x, \ \ \left(\g^-_x\right)^{-1} = \g^+_x.
\]
Moreover $c(\eta,G) \leq |\eta| + C |\L|$, where $|\L|$ is the Lebesgue measure of $\L$; therefore $\pi[c^2(\eta,G)] < +\infty$.

Now we define
\begin{eqnarray}
r(\eta,\g^+_x,\g^+_y) & = & \frac{dc(\g^+_x \eta, \cdot)}{dc( \eta, \cdot)}(\g^+_y) \ = \ \exp\left[- \b \nabla^+_x \nabla^+_y H(\eta)\right] \ = \ \exp\left[-\b\varphi(x-y)\right] \nonumber \\
r(\eta,\g^-_x,\g^-_y) & = & \frac{dc(\g^-_x \eta, \cdot)}{dc( \eta, \cdot)}(\g^-_y)  \ = \ \begin{cases} 1 & \mbox{for } x,y \in \eta, x \neq y \\ 0 & \mbox{otherwise} \end{cases} \label{rglac} \\
r(\eta,\g_x^-,\g_y^+) & = & r(\eta,\g_x^+,\g_y^-) \ = \ 1. \nonumber
\end{eqnarray}
\bl{lemma:glac1}
The function $r$ is admissible.
\el
\bpr
Note that the set $\{(\eta,\g,\d): \g(\d(\eta)) = \d(\g(\eta))\}$ has full measure for $ c(\eta,d\g)c(\eta,d\d)\pi(d\eta)$. Indeed, the only exception to commutativity $\g \circ \d (\eta) = \d \circ \g (\eta)$ is for $\g = \g_x^-$, $\d = \g_x^+$, $x \not\in \eta$; but it is easily seen that the set
\[
\{ (\eta,\g,\d): \, \exists x \not\in \eta \mbox{ such that } \g = \g_x^-, \, \d = \g_x^+\}
\]
is null for $ c(\eta,d\g)c(\eta,d\d)\pi(d\eta)$. Moreover, the symmetry condition $r(\eta,\g,\d) = r(\eta,\d,\g)$ is clear by definition of $r$. Thus, it is enough to prove \eqref{admissibility}. First, let $\g = \g_x^+$. Then 
\begin{multline} \label{adm1}
\int c(\eta, d\d)r(\eta,\g_x^+,\d)F(\d)  = \sum_{y \in \eta} r(\eta,\g_x^+,\g_y^-)F(\g_y^-) + z \int_{\L} e^{-\b\nabla^+_y H(\eta)}r(\eta,\g_x^+,\g_y^+) F(\g^+_y)\,dy \\ 
= \sum_{y\in\eta} F(\g_y^-) + z \int_{\L} e^{-\b\nabla^+_y H(\g_x^+(\eta))}F(\g^+_y)\,dy.
\end{multline}
Similarly
\begin{multline} \label{adm2}
\int c(\g^+_x(\eta), d\d) r(\g^+_x(\eta), \left(\g^+_x\right)^{-1},\d) F(\d) \\ = \sum_{y \in\g^+_x(\eta)} r(\g^+_x(\eta),\g_x^-,\g_y^-)F(\g_y^-) + z \int_{\L} e^{-\b\nabla^+_y H(\g^+_x(\eta))}r(\g^+_x(\eta),\g_x^-,\g_y^+) F(\g^+_y)dy \\ =  \sum_{y\in\eta} F(\g_y^-) +z \int_{\L} e^{-\b\nabla^+_y H(\g_x^+(\eta))}F(\g^+_y)dy,
\end{multline}
which shows \eqref{admissibility} for this case. The case $\g = \g_x^-$ is dealt with similarly.
\epr

\bt{th:glac}
Let $\e(\b)$ be the quantity defined in (\ref{integrability}) and assume $z \e(\b)<1$. Then
inequality \eqref{MLSI'} holds for 
\[
\kappa = 1-z\e(\b).
\]
Thus, for $z\e(\b)<1$, the entropy decays exponentially with a rate which is uniformly positive in  $\L$ and in the boundary condition $\tau$.
\et
\bpr
It is enough to prove \eqref{final}. First observe that, by ({\bf Rev}) and  \eqref{dirich},
\be{dirich1}
\cE(f,g) = \pi\left[ \sum_{x} \nabla_{x}^- f \nabla_{x}^- g\right] =z \int_{\L}\pi\left[ e^{-\b\nabla^+_x H} \nabla^+_x f \nabla^+_x g \right]dx.
\end{equation}
We have
\begin{multline} \label{glac-dec}
\int \Gamma(d\eta,d\g,d\d)  \left( \nabla_{\g} f(\eta) \nabla_{\d}\log  f(\eta) + \frac{ \nabla_{\g} f(\eta) \nabla_{\d} f(\eta)}{f(\eta)} \right) \\ = 
\int\pi\left[c(\cdot,d\g)c(\cdot,d\d) [1-r(\cdot,\g,\d)] \left( \nabla_{\g} f \nabla_{\d}\log  f + \frac{ \nabla_{\g} f \nabla_{\d} f}{f} \right)\right] \\
= \pi\left[ \sum_{x} \nabla_{x}^- f \nabla_{x}^- \log f\right] + \pi\left[ \sum_{x} \frac{\left( \nabla_{x}^-f\right)^2}{f} \right] \\
+ z^2\int_{\L^2}\pi\left[e^{-\b \nabla^+_x H} e^{-\b \nabla^+_y H} \left(1-e^{-\b \varphi(x-y)}\right)\left( \nabla^+_x f \nabla^+_y \log f + \frac{ \nabla^+_x f \nabla^+_y  f}{f} \right)\right]dx\,dy.
\end{multline}
By \eqref{dirich1}, the first summand in the r.h.s. of \eqref{glac-dec} equals $\cE(f,\log f)$. For the third summand we use \eqref{ineqkey}, together with the facts that, being $\varphi \geq 0$, we have $e^{-\b \nabla^+_y H(\eta)} \leq 1$ and $1-e^{-\b \varphi(x-y)} \geq 0$:
\begin{multline} \label{glac-dec2}
z^2 \int_{\L^2}\pi\left[e^{-\b \nabla^+_x H} e^{-\b \nabla^+_y H} \left(1-e^{-\b \varphi(x-y)}\right)\left( \nabla^+_x f \nabla^+_y \log f + \frac{ \nabla^+_x f \nabla^+_y  f}{f} \right)  \right]dx\,dy \\ 
\geq - z^2 \int_{\L^2}\pi\left[ e^{-\b \nabla^+_x H} e^{-\b \nabla^+_y H} \left(1-e^{-\b \varphi(x-y)}\right)\left( \nabla^+_x f \nabla^+_x \log f +\frac{\left( \nabla_x^+ f \right)^2}{f\circ \g_x^+}\right) \right]dx \\
\geq - z^2 \e(\b)  \int_{\L}\pi\left[e^{-\b \nabla^+_x H}\nabla^+_x f \nabla^+_x \log f  \right]dx
  - z^2 \e(\b)  \int_{\L}\pi\left[e^{-\b \nabla^+_x H} \frac{\left(\nabla_x^+ f \right)^2}{f\circ \g_x^+} \right]dx.
 \end{multline}
 Since
 \[
 z\int_{\L} \pi\left[e^{-\b \nabla^+_x H(\eta)}\nabla^+_x f \nabla^+_x \log f  \right]dx = \cE(f,\log f),
 \]
 and, by reversibility,
\[
 z\int_{\L}\pi\left[  e^{-\b \nabla^+_x H} \frac{\left(\nabla_x^+ f \right)^2}{f\circ \g_x^+} \right]dx  = \pi\left[ \sum_{x} \frac{\left( \nabla_{x}^-f\right)^2}{f} \right],
 \]
 by \eqref{glac-dec} and \eqref{glac-dec2} we obtain
 \begin{multline*}
 \int \Gamma(d\eta,d\g,d\d)  \left( \nabla_{\g} f(\eta) \nabla_{\d}\log  f(\eta) + \frac{ \nabla_{\g} f(\eta) \nabla_{\d} f(\eta)}{f(\eta)} \right) \\
 \geq 
 (1-z\e(\b)) \cE(f,\log f) + (1-z\e(\b)) \pi\left[ \sum_{x} \frac{\left( \nabla_{x}^-f\right)^2}{f} \right] \geq (1-z\e(\b)) \cE(f,\log f),
  \end{multline*}
 which completes the proof of \eqref{final}.
 
\epr

\br{rem:gap}
Theorem \ref{th:glac} provides the lower bound $\a \geq 1-z\e(\b)$ for the best constant $\a$ in the entropy inequality. Note that it coincides with the lower bound, obtained e.g. in \cite{Bo:Ca:DP:Po}, for the spectral gap $\g$. The upper bound $\g \leq 1+z\e(\b)$ has also be obtained in \cite{wu2}.

\er

\subsection{Interacting birth and death processes and a simple non perturbative example}

With no essential change, the arguments in Section \ref{sec:gas} can be adapted to the following discrete version of the model, which can be viewed as describing  a family of interacting birth and death processes.
Let $h:\integer^d\to[0,+\infty)$ be such that $h(0)=0$, $h(-x)=h(x)$ and 
\begin{displaymath}
  \sum_{x\in\integer^d}h(x)<+\infty.
\end{displaymath}
Define $\Lambda_L:=\integer^d\cap[1,L]^d$ and consider for $\eta\in S:=\Omega_L:=\{\eta: \Lambda_L\to\natural\cup\{0\}\}$, the \emph{pair potential}:
\begin{displaymath}
  \varphi(x,y,\eta):=h(x-y)\eta(x)\eta(y),
\end{displaymath}
and the Hamiltonian (a boundary condition can be added as is section~\ref{sec:cont})
\begin{equation}
  \label{eq:Ham}
  H(\eta):=\frac{1}{2}\sum_{x,y\in\Lambda_L}\varphi(x,y,\eta).
\end{equation}
We assume the function $h$ to satisfy the condition
\begin{displaymath}
  \epsilon(\beta):=
  \sum_{x\in\integer^d}\left(1-e^{-\beta h(x)}\right)<+\infty
\end{displaymath}
for $0\leq\beta<\beta_0$.
The finite volume grand canonical Gibbs measure $\pi$ with inverse temperature $\beta$ and activity $z$ is the probability measure defined on $S$ as
\begin{displaymath}
  \pi(\eta):=
  \frac{1}{Z}e^{-\beta H(\eta)}\prod_{x\in\Lambda_L}\frac{z^{\eta(x)}}{\eta(x)!},
\end{displaymath}
where $Z$ is the normalization.
Fix $x\in \integer^d$; given any configuration $\eta\in S$ we define $\eta\pm\delta_x$ as $(\eta\pm\delta_x)(y):=\eta(y)\pm\ind(x=y)$.
Define also the creation and annihilation maps at $x$, $\gamma_x^\pm: S\to S$, as
\begin{displaymath}
  \gamma_x^+(\eta):=\eta+\delta_x,
  \qquad\qquad
  \gamma_x^-(\eta):=
  \begin{cases}
    \eta-\delta_x & \text{if $\eta(x)>0$} \\
    \eta  & \text{otherwise.}
  \end{cases}
\end{displaymath}
We let $G:=\{\gamma_x^-,\gamma_x^+:x\in T\}$.
In the sequel we write $\nabla_x^+$ and $\nabla_x^-$ rather than $\nabla_{\gamma_x^+}$ and $\nabla_{\gamma_x^-}$.
We consider the Markov generator
\begin{displaymath}
  \mathcal{L}f(\eta):=
  \sum_{x\in\Lambda_L}\left[\eta(x)\nabla_x^-f(\eta)+ze^{-\beta\nabla_x^+H(\eta)}\nabla_x^+f(\eta)\right].
\end{displaymath}
It is easy to show that $\mathcal{L}$ is self adjoint in $L^2(\pi)$, and that generates a Markov semigroup.
It can be written in the form \eqref{gen} by defining $c(\eta,d\gamma)$ analogously to section~\ref{sec:cont}.
In particular, the condition $\pi[c^2(\eta,G)] < +\infty$ is satisfied.
By defining the admissible function $r:S\times G\times G\to \real$,
\begin{eqnarray*}
r(\eta,\g^+_x,\g^+_y) & = & \frac{dc(\g^+_x \eta, \cdot)}{dc( \eta, \cdot)}(\g^+_y) \ = \ \exp\left[- \b \nabla^+_x \nabla^+_y H(\eta)\right]  \\
r(\eta,\g^-_x,\g^-_y) & = & \frac{dc(\g^-_x \eta, \cdot)}{dc( \eta, \cdot)}(\g^-_y)  \ = \ \begin{cases}
  \frac{\eta(x)-1}{\eta(x)} & \text{if $x=y$ and $\eta(x)>0$,} \\
  1 & \text{otherwise}
\end{cases}  \\
r(\eta,\g_x^-,\g_y^+) & = & r(\eta,\g_x^+,\g_y^-) \ = \ 1, 
\end{eqnarray*}
and following the same arguments of section~\ref{sec:cont}, it can be shown that Theorem~\ref{th:glac} holds also in this case.

The condition $z\epsilon(\beta)<1$, under which the convex exponential decay of entropy has been established in both the continuous and discrete space, is a \emph{high temperature/low density condition}, \ie a condition which states that the measure $\pi$ and the associated dynamics generated by $\mathcal{L}$ are  small perturbations of a system of independent particles,  for which \eqref{MLSI} holds by standard tensorization properties.

It is interesting to observe that the same technique can be applied to cases which are \emph{far} from a product case, by requiring some convexity on the Hamiltonian $H$.
This is quite natural in the $\Gamma_2$ approach (see \cite{Ba:Em}). However, the nonlocality of the generators is a source of serious limitations. The main problem is the fact that  inequality~\eqref{key} is only bivariate: rather surprisingly, ``natural'' multivariate extensions of it are false. This forces us to consider systems of only {\em two} interacting birth and death processes.

In the notations of the present section choose $d=1$, $L=2$, $z=1$, $H(\eta)=K(\eta_1+\eta_2)$, with $K$ an increasing convex function (\eg $K(u)=u^2$).
Notice that under these  conditions $\nabla_{1}^+H=\nabla_{2}^+H\geq0$ and $\nabla_1^+\nabla_1^+H=\nabla_1^+\nabla_2^+H=\nabla_2^+\nabla_2^+H\geq0$.
As in the proof of Theorem~\ref{th:glac} it can be shown that, for $f>0$ with $\log f$ bounded,
\begin{multline*} \label{glac-dec}
  \int \Gamma(d\eta,d\gamma,d\delta)  \left( \nabla_{\gamma} f(\eta) \nabla_{\delta}\log  f(\eta) + \frac{ \nabla_{\gamma} f(\eta) \nabla_{\delta} f(\eta)}{f(\eta)} \right)\\ 
  =\sum_{x=1}^2\pi\left[\eta(x)\left\{\nabla_x^- f \nabla_x^- \log f+\frac{\left( \nabla_x^-f\right)^2}{f}\right\}\right] \\
  +\sum_{x,y=1}^2\pi\left[e^{-\beta \nabla^+_x H} e^{-\beta \nabla^+_y H} \left(1-e^{-\beta\nabla_x^+\nabla_y^+H}\right)\left( \nabla^+_x f \nabla^+_y \log f + \frac{ \nabla^+_x f \nabla^+_y  f}{f} \right) \right].
\end{multline*}
By erasing a positive term, using reversibility and symmetrizing, we get
\begin{multline*}
  \int \Gamma(d\eta,d\gamma,d\delta)  \left( \nabla_{\gamma} f(\eta) \nabla_{\delta}\log  f(\eta)
    + \frac{ \nabla_{\gamma} f(\eta) \nabla_{\delta} f(\eta)}{f(\eta)} \right)\\
  \geq\mathcal{E}(f,\log f)+\sum_{x=1}^2\pi\left[\eta(x)\frac{\left( \nabla_x^-f\right)^2}{f}\right]+
  \sum_{x=1}^2\pi\left[e^{-2\beta \nabla^+_x H}\left(1-e^{-\beta\nabla_x^+\nabla_x^+H}\right)\nabla^+_x f \nabla^+_x \log f \right]\\
  +\frac{1}{2}\sum_{x\not=y}\pi\left[e^{-\beta \nabla^+_x H}e^{-\beta \nabla^+_y H}\left(1-e^{-\beta\nabla_x^+\nabla_y^+H}\right)\left( \nabla^+_x f \nabla^+_y \log f+\nabla^+_y f \nabla^+_x \log f + 2\frac{ \nabla^+_x f\nabla^+_y f}{f} \right) \right]
\end{multline*}

Using \eqref{ineqkey} and reversibility on the last term we obtain:
\begin{multline*}
  \frac{1}{2}\sum_{x\not=y}\pi\left[e^{-\beta \nabla^+_x H}e^{-\beta \nabla^+_y H}\left(1-e^{-\beta\nabla_x^+\nabla_y^+H}\right)\left( \nabla^+_x f \nabla^+_y \log f+\nabla^+_y f \nabla^+_x \log f + 2\frac{ \nabla^+_x f\nabla^+_y f}{f} \right) \right]\\
  \geq-\sum_{x=1}^2\pi\left[e^{-\beta \nabla^+_1 H}e^{-\beta \nabla^+_2 H}\left(1-e^{-\beta\nabla_1^+\nabla_2^+H}\right)\left\{ \nabla^+_x f \nabla^+_x \log f+ \frac{ (\nabla^+_x f)^2}{f\circ\gamma_x^+} \right\} \right]\\
  =-\sum_{x=1}^2\pi\left[e^{-2\beta \nabla^+_x H}\left(1-e^{-\beta\nabla_x^+\nabla_x^+H}\right)\nabla^+_x f \nabla^+_x \log f \right]\\
  -\sum_{x=1}^2\pi\left[\eta(x)e^{ -\beta(\nabla_x^+H)\circ\gamma_x^-}\left\{1-e^{-\beta (\nabla_x^+\nabla_x^+H)\circ\gamma_x^-}\right\}\frac{ (\nabla_x^- f)^2}{f}\right].
\end{multline*}
So we have that
\begin{multline*}
  \int \Gamma(d\eta,d\gamma,d\delta)  \left( \nabla_{\gamma} f(\eta) \nabla_{\delta}\log  f(\eta)
    + \frac{ \nabla_{\gamma} f(\eta) \nabla_{\delta} f(\eta)}{f(\eta)} \right)\geq
  \\
  \mathcal{E}(f,\log f)
  +\sum_{x=1}^2\pi\left[\eta(x)\frac{\left( \nabla_x^-f\right)^2}{f}\right]
  -\sum_{x=1}^2\pi\left[\eta(x)e^{-\beta (\nabla_x^+H)\circ\gamma_x^-}\left\{1-e^{-\beta(\nabla_x^+\nabla_x^+H)\circ\gamma_x^-}\right\}\frac{ (\nabla_x^- f)^2}{f}\right].
\end{multline*}
We can conclude that inequality \eqref{MLSI'} holds with $\kappa=1$ for any $\beta\geq0$ by observing that, under the current assumptions  on $H$:
\begin{displaymath}
  \eta(x)e^{-\beta\nabla_x^+H(\eta-\delta_x)}\left\{1-e^{-\beta\nabla_x^+\nabla_x^+H(\eta-\delta_x)}\right\}
  \leq \eta(x)
\end{displaymath}
for any $\eta\in S$, $x\in\{1,2\}$.

\subsection{A general hardcore model}
\label{sec:hc}

In this section we present a general birth and death process taking values in the set of multi-subsets of a given finite set. 
While it is possible, with minimum effort, to establish similar results for more general interactions, we limit our analysis to models where the interaction takes the form of a general {\em exclusion rule}.

Let $T$ be a finite set and consider the configuration space $S:=\{\eta:T\to\natural\cup\{0\}\}$.
In $S$ there is a natural (partial) order relation defined by $\eta,\xi\in S$, $\eta\leq\xi$ if and only if $\eta(x)\leq\xi(x)$ for any $x\in T$.
A \emph{decreasing} subset $A$ of $S$ is an $A\subseteq S$ with the property that given $\eta\in S$ and $\xi\in A$ with $\eta\leq\xi$ then $\eta\in A$.

Fix a decreasing $A\subseteq S$ as the set of \emph{allowed configuration} and an \emph{intensity} $\nu:T\to(0,+\infty)$.
We can define the probability measure on $S$ given by
\begin{displaymath}
  \pi(\eta):=
  \frac{\ind(\eta\in A)}{Z}\prod_{x\in T}\frac{\nu(x)^{\eta(x)}}{\eta(x)!},
\end{displaymath}
where $Z$ is the normalization.

We are going to define a Markov chain on $S$ reversible with respect to $\pi$.
Fix $x\in T$, given any configuration $\eta\in S$ we define $\eta+\delta_x$ and $\eta-\delta_x$ as $(\eta\pm\delta_x)(y):=\eta(y)\pm\ind(x=y)$.
Define also the creation and annihilation maps at $x$, $\gamma_x^\pm: A\to A$ as
\begin{displaymath}
  \gamma_x^+(\eta):=
  \begin{cases}
    \eta+\delta_x & \text{if $\eta+\delta_x\in A$} \\
    \eta  & \text{otherwise,}
  \end{cases}
  \qquad\qquad
  \gamma_x^-(\eta):=
  \begin{cases}
    \eta-\delta_x & \text{if $\eta-\delta_x\in A$} \\
    \eta  & \text{otherwise.}
  \end{cases}
\end{displaymath}
We let $G:=\{\gamma_x^-,\gamma_x^+:x\in T\}$.
In the sequel we write $\nabla_x^+$ and $\nabla_x^-$ rather than $\nabla_{\gamma_x^+}$ and $\nabla_{\gamma_x^-}$.
Observe that $\nabla_x^-f(\eta)=0$ if $\eta(x)=0$ and $\nabla_x^+f(\eta)=0$ if $\eta+\delta_x\not\in A$.
Consider now the Markov generator
\begin{equation}
  \label{eq:gln}
  \mathcal{L}f(\eta)=
  \sum_{x\in T}\left[\eta(x)\nabla_x^-f(\eta)+\nu(x)\nabla_x^+f(\eta)\right].
\end{equation}
It is easy to check that $\mathcal{L}$ is self-adjoint in $L^2(\pi)$,  it can be written in the form \eqref{gen},
with $\pi[c^2(\eta,G)] < +\infty$.

Now we define
\begin{eqnarray}
r(\eta,\g^+_x,\g^+_y) & = & \frac{dc(\gamma^+_x(\eta), \cdot)}{dc( \eta, \cdot)}(\gamma^+_y) \ = \ \ind(\eta+\delta_x+\delta_y\in A) \nonumber \\
r(\eta,\g^-_x,\g^-_y) & = & \frac{dc(\gamma^-_x(\eta), \cdot)}{dc( \eta, \cdot)}(\gamma^-_y)  \ = \
\begin{cases}
  \frac{\eta(x)-1}{\eta(x)} & \text{if $x=y$ and $\eta(x)>0$,} \\
  1 & \text{otherwise}
\end{cases}
\label{rln} \\
r(\eta,\g_x^-,\g_y^+) & = & r(\eta,\g_x^+,\g_y^-) \ = \ 1. \nonumber
\end{eqnarray}
It is elementary to check that $r$ is \emph{admissible}.
This allows us to prove the following result.
\begin{theorem}
  \label{teo:hc}
  Define
\begin{align*}
  \epsilon_0&:=
  \sup_{x,\eta:\eta(x)>0}\sum_{y:y\not=x}\nu(y)\ind(\eta-\delta_x+\delta_y\in A)\ind(\eta+\delta_y\not\in A)\\
  \epsilon_1&:=
  \inf_{x,\eta:\eta(x)>0}\nu(x)\ind(\eta+\delta_x\not\in A),
\end{align*}
and assume $\epsilon_0  \leq 1$. Then inequality \eqref{MLSI'} holds for $\kappa=1-\epsilon_0+\epsilon_1$.
\end{theorem}

\begin{proof}
  Observe that
  \begin{align*}
    1-r(\eta,\gamma_x^+,\gamma_y^+)&=\ind(\eta+\delta_x+\delta_y\not\in A) \\
    1-r(\eta,\gamma_x^-,\gamma_y^-)&=\frac{\ind(x=y,\eta(x)>0)}{\eta(x)} \\
    1-r(\eta,\gamma_x^-,\gamma_y^+)&=1-r(\eta,\gamma_x^+,\gamma_y^-)=0.
  \end{align*}
Thus, for $f>0$ with $\log f$ bounded,
\begin{multline*} 
\int \Gamma(d\eta,d\gamma,d\delta)  \left( \nabla_{\gamma} f(\eta) \nabla_{\delta}\log  f(\eta) + \frac{ \nabla_{\gamma} f(\eta) \nabla_{\delta} f(\eta)}{f(\eta)} \right) \\
= 
\int\pi\left[  c(\cdot,d\gamma)c(\cdot,d\delta) [1-r(\cdot,\g,\delta)] \left( \nabla_{\gamma} f \nabla_{\delta}\log  f + \frac{ \nabla_{\gamma} f \nabla_{\delta} f}{f} \right)\right] \\
= \sum_{x \in T}\pi\left[\eta(x)\left\{\nabla_{x}^- f \nabla_{x}^- \log f+\frac{\left( \nabla_{x}^-f\right)^2}{f}\right\} \right] \\
+\sum_{x \in T}\nu^2(x)\pi\left[\ind(\cdot+\delta_x\in A)\ind(\cdot+2\delta_x\not\in A)\left\{\nabla_{x}^+ f \nabla_{x}^+ \log f+\frac{\left( \nabla_{x}^+f\right)^2}{f}\right\} \right]+
\\
+\sum_{x\not=y}\nu(x)\nu(y)\pi
\left[
  \ind(\cdot+\delta_x\in A)\ind(\cdot+\delta_y\in A)\ind(\cdot+\delta_x+\delta_y\not\in A)
  \left\{
    \nabla_{x}^+ f \nabla_{y}^+ \log f+
    \frac{\nabla_{x}^+f\nabla_{y}^+f}{f}
  \right\}
\right].
\end{multline*}
Following the by now usual steps, using reversibility and symmetrization, we obtain
\begin{multline*} 
\int \Gamma(d\eta,d\gamma,d\delta)  \left( \nabla_{\gamma} f(\eta) \nabla_{\delta}\log  f(\eta) + \frac{ \nabla_{\gamma} f(\eta) \nabla_{\delta} f(\eta)}{f(\eta)} \right)\geq \\  
\geq\mathcal{E}(f,\log f) +\sum_{x \in T}\pi\left[\eta(x)\frac{\left( \nabla_{x}^-f\right)^2}{f} \right]+\sum_{x \in T}\nu(x)\pi\left[\eta(x)\ind(\cdot+\delta_x\not\in A)\nabla_{x}^+ f \nabla_{x}^+ \log f \right]
\\
+\frac{1}{2}\sum_{x\not=y}\nu(x)\nu(y)\pi
\Bigg[
  \ind(\cdot+\delta_x\in A)\ind(\cdot+\delta_y\in A)\ind(\cdot+\delta_x+\delta_y\not\in A)\\
  \times\Bigg(
    \nabla_{x}^+ f \nabla_{y}^+ \log f+
    \nabla_{y}^+ f \nabla_{x}^+\log f+
    2\frac{\nabla_{x}^+f\nabla_{y}^+f}{f}
  \Bigg)
\Bigg].
\end{multline*}
Finally, by  \eqref{ineqkey}
\begin{multline*}
  \frac{1}{2}\sum_{x\not=y}\nu(x)\nu(y)\pi
\Bigg[
  \ind(\cdot+\delta_x\in A)\ind(\cdot+\delta_y\in A)\ind(\cdot+\delta_x+\delta_y\not\in A)\\
  \times\Bigg(
    \nabla_{x}^+ f \nabla_{y}^+ \log f+
    \nabla_{y}^+ f \nabla_{x}^+\log f+
    2\frac{\nabla_{x}^+f\nabla_{y}^+f}{f}
  \Bigg)
\Bigg]\\
\geq -\sum_{x\not=y}\nu(x)\nu(y)\pi\left[\ind(\cdot+\delta_x\in A)\ind(\cdot+\delta_y\in A)\ind(\cdot+\delta_x+\delta_y\not\in A)\left(\nabla_{x}^+ f \nabla_{x}^+\log f+
    \frac{(\nabla_{x}^+f)^2}{f\circ\gamma_x^+}\right)\right],
\end{multline*}
and using reversibility on the last term we obtain
\begin{multline*}
  \frac{1}{2}\sum_{x\not=y}\nu(x)\nu(y)\pi
\Bigg[
  \ind(\cdot+\delta_x\in A)\ind(\cdot+\delta_y\in A)\ind(\cdot+\delta_x+\delta_y\not\in A)\\
  \times\Bigg(
    \nabla_{x}^+ f \nabla_{y}^+ \log f+
    \nabla_{y}^+ f \nabla_{x}^+\log f+
    2\frac{\nabla_{x}^+f\nabla_{y}^+f}{f}
  \Bigg)
\Bigg]\\
\geq
-\sum_{x\not=y}\nu(y)\pi\left[\eta(x)\ind(\cdot-\delta_x+\delta_y\in A)\ind(\cdot+\delta_y\not\in A)\left(\nabla_{x}^- f \nabla_{x}^-\log f+
    \frac{(\nabla_{x}^-f)^2}{f}\right)\right]\\
\geq-\epsilon_0\left\{\mathcal{E}(f,\log f)+\sum_{x\in T}\pi\left[\eta(x)\frac{(\nabla_x^-f)^2}{f}\right]\right\}.
\end{multline*}

Observing that
\[
\sum_{x \in T}\nu(x)\pi\left[\eta(x)\ind(\cdot+\delta_x\not\in A)\nabla_{x}^+ f \nabla_{x}^+ \log f \right] \geq \epsilon_1 \mathcal{E}(f,\log f),
\]
all this sums up to
\begin{multline*} 
\int \Gamma(d\eta,d\gamma,d\delta)  \left( \nabla_{\gamma} f(\eta) \nabla_{\delta}\log  f(\eta) + \frac{ \nabla_{\gamma} f(\eta) \nabla_{\delta} f(\eta)}{f(\eta)} \right)\\  
\geq(1-\epsilon_0+\epsilon_1)\mathcal{E}(f,\log f) +(1-\epsilon_0)\sum_{x \in T}\pi\left[\eta(x)\frac{\left( \nabla_{x}^-f\right)^2}{f} \right].
\end{multline*}
\end{proof}

In the next examples we give some application of Theorem~\ref{teo:hc}.

\subsubsection{The hardcore model}
\label{ex:hm}
Let $G=(V,E)$ a finite, connected (symmetric, simple) graph (we let $E\subseteq\{\{x,y\}\subseteq V: x\not=y\}$).
Take $T:=V$, $\nu\equiv\rho>0$ and
\begin{displaymath}
  A:=\left\{\eta\in S: \text{$\eta(x)\in\{0,1\}$ for any $x\in V$ and  $\eta(x)\eta(y)=0$ if $\{x,y\}\in E$}\right\}.
\end{displaymath}
Then define the \emph{maximum degree of $G$} as $\Delta:=\max_{x\in V}\deg(x,V)=\max_{x\in V}\sum_{y\in V}\ind(\{x,y\}\in E)$.
We have that $\epsilon_0=\Delta\rho$ and $\epsilon_1=\rho$.
This gives  $\kappa\geq1-\rho(\Delta-1)$ for $\rho\leq1/\Delta$, \ie the mixing time does not depend on the size of the graph provided that $\rho\leq1/\Delta$.

The hardcore model has been widely studied in literature (see \cite{Le:Pe:Wi} section~22.4 and the discussion therein).
The best result on the mixing time for this model on general graph known by authors is the fast mixing result for $\rho<2/(\Delta-2)$ contained in \cite{Lu:Vi,Vi}.
We want to stress that the model considered in \cite{Lu:Vi,Vi} is a discrete time Markov chain which can be compared with our result by using Theorem~20.3 of \cite{Le:Pe:Wi}.

\subsubsection{Loss Networks}
\label{ex:ln}

For a complete introduction to loss networks we refer to \cite{Ke}.
Here we give only a brief sketch of the model.

Consider a finite, connected (symmetric, simple) graph $G=(V,E)$ and a function $C:E\to\natural\cup\{+\infty\}$  called \emph{capacity function}.
A \emph{path} in $G$ is a sequence $(e_1,\dots,e_n)$ of edges in $E$ such that $e_i\cap e_{i+1}\not=\emptyset$, $i=1,\dots,n$ and $e_i\not=e_j$ for any $i\not=j$.
Given a path $x=(e_1,\dots,e_n)$ and an edge $e\in E$ we say that $e$ belongs to $x$ if $e=e_i$ for some $i\in\{1,\dots,n\}$.
We write $e\in x$ in this case.
Let $T$ be a collection of paths in $G$.
A configuration  $\eta$ is an element of $S:=\{\eta:T\to\natural\cup\{0\}\}$.
$G$ should be thought as the graph of a ``telecommunication network'' in which $v\in V$ are ``callers'' and $e\in E$ are ``links''.
$T$ represents the set of possible ``routes'' which a call can use to connect two callers.
For $\eta\in S$ and $x\in T$, $\eta(x)$ is the number of routes of type $x\in T$.
So, given $\eta\in S$ the number of calls using the link $e\in E$ is $\sum_{x\ni e}\eta(x)$.
We impose that there are at most $C(e)$ calls using the link $e$ by requiring that the set of allowed route is given by the decreasing set
\begin{displaymath}
  A=
  \left\{\eta\in S: \sum_{x\ni e}\eta(x)\leq C(e) \text{ for any $e\in E$}\right\}.
\end{displaymath}
Now fix an \emph{intensity function} $\nu:T\to(0,+\infty)$. 
The generator given by \eqref{eq:gln} is the generator a Markov chain in which calls arrive independently with intensity $\nu$.
If a call which violates the constraint defined by $A$ arrives, it is rejected.
Any call lasts for an exponential time of mean 1.
Is should be clear that $\epsilon_0$ is small if $\max_{x\in T}\nu(x)$ is small enough (depending on the geometry of $G$, $T$ and on the function $C$).
So we can get lower bound on $\kappa$ by taking small intensities.

\subsubsection{Long hard rods}
\label{ex:hr}
This is a statistical mechanics model for liquid crystals.
See \cite{Di:Gi} for a deeper discussion of the model. 
Let $L,k\in\natural$, with $L\gg k$.
Consider the graph $G:=(V,E)$ where $V:=\integer^2\cap[0,L]^2$ and
 $E:=\{\{(u_1,u_2),(v_1,v_2)\}\subseteq V:(u_1-v_1)^2+(u_2-v_2)^2=1\}$.
An \emph{horizontal rod} of length $k$ is a sequence of $k+1$ adjacent
vertexes of $V$ in ``horizontal'' direction
\begin{displaymath}
  \{(u_1,u_2),(u_1+1,u_2),\dots,(u_1+k,u_2)\}.
\end{displaymath}
Denote by $T_+$ the set of horizontal rods of length $k$. 
Similarly a \emph{vertical rod} of length $k$ is a sequence of $k+1$ adjacent
vertexes of $V$ in ``vertical'' direction
\begin{displaymath}
  \{(u_1,u_2),(u_1,u_2+1),\dots,(u_1,u_2+k)\}.
\end{displaymath}
Denote by $T_-$ the set of vertical rods of length $k$.
We set $T=T_+\cup T_-$,
\begin{displaymath}
  A:=\{\eta\in S: \eta(x)\in\{0,1\},\  \eta(x)\eta(y)=0 \text{ if }  x\not= y \text{ and } x\cap y\not=\emptyset\}
\end{displaymath}
(\ie rods can not touch), and $\nu\equiv\rho>0$.
We then obtain $\epsilon_0=\rho (k^2+4k+1)$, $\epsilon_1=\rho$ so that if $\rho\leq1/(k^2+4k+1)$ then $\kappa\geq1-\rho(k^2+4k)$.
We recall that for $k$ sufficiently large (see \cite{Di:Gi}) the Gibbs measure $\pi$ exhibits a phase transition in the limit $L\to+\infty$ at some point $\rho_c$ which is expected to be of order $1/k^2$.
We therefore obtain the exponential decay of entropy for $\rho$ up to $1/k^2$, which has the same order in $k$ as the critical value $\rho_c$.

\section{Appendix: Proof of Proposition \ref{prop:conv-ent}}

We first observe that $\A_M$ is $T_t$-stable, i.e. if $f \in \A_M$, than $T_t f \in \A_M$ too. The implication
\[
e^{-M} \leq f \leq e^M \ \Rightarrow \ e^{-M} \leq T_t f \leq e^M,
\]
follows from positivity of $T_t$. Moreover, by a standard application of functional calculus, if $f \in \cD(\cL^2)$ then $T_t f \in  \cD(\cL^2)$ and $\cL^2 T_t f = T_t \cL^2 f$.
\ben
\item
Assume ({\bf EI}) holds for every $f \in \A$. By definition of generator, if $f \in \cD(\cL^2) \subseteq  \cD(\cL)$, then $T_t f$ is differentiable in the $L^2$ sense. We also claim that, if $f \in \A_M$, then $\log T_t f$ is $L^1$-differentiable, and
\[
\frac{d}{dt} \log T_t f := L^1-\lim_{h \ra 0} \frac{\log T_{t+h} f - \log T_t f}{h} = \frac{\cL T_t f}{T_t f}.
\]
Indeed, using the inequality $|\log(1+x) - x| \leq \frac{x^2}{1+x}$, we have
\[
\begin{split}
\left\| \frac{\log T_{t+h} f - \log T_t f}{h} - \frac{\cL T_t f}{T_t f} \right\|_1  & = \,  \left\| \frac{1}{h}\log\left(1+ \frac{h}{T_t f} \frac{T_{t+h}f - T_t f}{h} \right)  - \frac{\cL T_t f}{T_t f} \right\|_1 \\
& \leq \,  \left\| \frac{1}{h}\left[\log\left(1+ \frac{h}{T_t f} \frac{T_{t+h}f - T_t f}{h} \right) - \frac{h}{T_t f} \frac{T_{t+h}f - T_t f}{h}\right] \right\|_1 \\ & \, \  \ + \, 
\left\| \frac{ \frac{T_{t+h}f - T_t f}{h}}{T_t f} - \frac{\cL T_t f}{T_t f} \right\|_1 \\
& \leq \, \left\| \frac{1}{h} \frac{ \left(\frac{h}{T_t f}\right)^2 \left(\frac{T_{t+h}f - T_t f}{h} \right)^2}{T_{t+h} f /T_{t} f } \right\|_1 + \, 
\left\| \frac{ \frac{T_{t+h}f - T_t f}{h}}{T_t f} - \frac{\cL T_t f}{T_t f} \right\|_1 \\
& \leq |h| e^{2M}   \left\|\left(\frac{T_{t+h}f - T_t f}{h} \right)^2 \right\|_1 + e^M \left\|  \frac{T_{t+h}f - T_t f}{h} - \cL T_t f \right\|_1 \\
& \leq  |h| e^{2M}   \left\|\frac{T_{t+h}f - T_t f}{h} \right\|_2^2 +  e^M \left\|  \frac{T_{t+h}f - T_t f}{h} - \cL T_t f \right\|_2,
\end{split}
\]
and both last summands go to zero as $h \ra 0$.

Now, assuming without loss of generality that $\pi[f] =1$, we show that the expression
\[
\ent_{\pi}(T_t f) = \pi[T_t f \log T_t f]
\]
can be differentiated commuting derivative with expectation, obtaining
\be{tech3}
\frac{d}{dt} \ent_{\pi}(T_t f) = - \cE(T_t f,\log T_t f).
\end{equation}
Indeed
\[
\frac1h [\ent_{\pi}(T_{t+h} f) - \ent_{\pi}(T_t f)] =  \pi\left[ \frac{T_{t+h}f - T_t f}{h} \log T_{t+h} f \right] + \pi\left[ \frac{\log T_{t+h} f - \log T_t f}{h} T_t f \right].
\]
Since both $ \log T_{t+h} f$ and $T_t f$ are uniformly bounded, we can take the limit in the above expression, obtaining \eqref{tech3}. From \eqref{tech3}, using ({\bf EI}) and Gromwall's Lemma, we get \eqref{tech1} for every $f \in \A$. Now, for $f \in L^2_M$, by Assumption A there is a sequence $f_n \in \A_M$ which
converges to $f$ in $L^2(\pi)$, which implies $T_t f_n \ra T_t f$ in $L^2(\pi)$; note that also $\log f_n$ (resp. $\log T_t f_n$) converges to $\log f$ (resp. $\log T_t f$) in $L^2(\pi)$ (e.g. observe that $x \mapsto \log x$ is Lipschitz continuous in $[e^{-M}, e^M]$). Thus $\ent_{\pi}(f_n) = \pi[f_n \log f_n] - \pi[f_n] \log \pi[f_n] \ra \pi[f \log f] = \ent_{\pi}(f)$, which implies that \eqref{tech1} holds for $f \in \cup_M L^2_M$. Finally, set $f$ with $\ent_{\pi}(f) < +\infty$, and set 
\[
f_M := (f \wedge e^M)\vee e^{-M}.
\]
Clearly $f_M \in \A_M$, $f_M \ra f$ a.s. and in $L^1$ as $M \ra +\infty$. Moreover
\[
\left( f_M \log f_M \right)^+ \leq (f \log f)^+,
\]
while $\left( f_M \log f_M \right)^-$ is bounded uniformly. Thus, by dominated convergence,
\[
\ent_{\pi}(f_M) \ra \ent_{\pi}(f).
\]
Note that, possibly along subsequences, the same convergence holds for $T_t f_M$ and $T_t f$ in place of $f_M$ and $f$. This implies, by approximation, that \eqref{tech1} holds for every $f$ with $\ent_{\pi}(f) < +\infty$.

We now show the converse implication. Assume \eqref{tech1} holds for every $f \geq 0$ measurable, such that $\ent_{\pi}(f) < +\infty$. In particular, it holds for $f \in \A$. The functions of $t$ $\ent_{\pi}(T_t f)$ and $e^{-\a t} \ent_{\pi}(f)$ are both differentiable, and coincide at $t=0$. Necessarily
\[
\frac{d}{dt} \ent_{\pi}(T_t f) \big|_{t=0} \leq \frac{d}{dt}e^{-\a t} \ent_{\pi}(f)\big|_{t=0},
\]
which gives ({\bf EI}).

\item
Note that, so far, we have not used all properties of $\A_M$, but only the facts that $ f \in \cD(\cL)$ and $ |\log f| \leq M$. The other  properties are used below to take the second derivative of the entropy. 

Suppose \eqref{MLSI'} holds for all $f \in \A$. The point is to justify the differentiation
\[
\frac{d}{dt} \cE(T_t f,\log T_t f) = -  \frac{d}{dt} \pi\left[ T_t \cL f \, \log T_t f \right].
\]
Similarly to what we have done in point 1, we can differentiate using the product rule since:
\bi
\item
$t \mapsto T_t \cL f$ is $L^2$-differentiable for $f \in \A$, since $\cL f \in \cD(\cL)$, and $\log T_t f $ is uniformly bounded;
\item
$\log T_t f$ is $L^1$-differentiable, and $T_t \cL f$ is uniformly bounded.
\ei
We obtain
\be{tech4}
\frac{d}{dt} \cE(T_t f,\log T_t f) = 
- \pi\left[ \cL T_t f \cL \log T_t f\right] -  
\pi \left[ \frac{(\cL T_t f)^2}{T_t f} \right]
\end{equation}
that, by \eqref{MLSI'} and Gromwall's Lemma, implies $\cE(T_t f,\log T_t f) \leq e^{-\k t} \cE(f,\log f)$. Conversely, \eqref{MLSI'} follows from $\cE(T_t f,\log T_t f) \leq e^{-\k t} \cE(f,\log f)$ by taking derivatives at $t=0$, as in point 1.

\item
From \eqref{tech4}, \eqref{MLSI'} and \eqref{tech3}, we obtain
\[
- \frac{d}{dt} \cE(T_t f,\log T_t f) \geq -\k\frac{d}{dt} \ent_{\pi}(T_t f)\,.
\]
that, integrated from $0$ to $\infty$, yields
\[
k \,\ent_{\pi}(f) \leq  \cE(f, \log  f),
\]
which completes the proof.

\een

\subsection*{Acknowledgements.}
We thank P. Caputo for useful discussions.

\end{document}